\documentclass[12pt]{amsart}
\usepackage{times}
\usepackage{amsfonts}
\usepackage{amssymb}
\usepackage{bbm}
\usepackage{times}
\usepackage{amssymb}
\usepackage{amscd}
\usepackage{graphicx}

\usepackage{amsmath}
\usepackage{amssymb}

\usepackage{amsmath}
\usepackage{amsfonts}
\usepackage{amscd}
\usepackage{latexsym}
\usepackage{amsthm}
\usepackage{amssymb}
\usepackage{amsmath}
\usepackage{amssymb, enumerate}

\usepackage{amsmath}
\usepackage{amssymb}
\usepackage{amsthm}

\theoremstyle{plain}
\newtheorem{theorem}{Theorem}[section]

\newtheorem{lemma}[theorem]{Lemma}

\newtheorem{definition}[theorem]{Definition}

\begin{document}

\title[Extension of $C^*$-algebras] {certain Cuntz semigroup properties of extension C*-algebras}
\author{Qingzhai Fan}
\address{Qingzhai Fan\\ Department of Mathematics\\  Shanghai Maritime University\\
Shanghai\\China
\\  201306 }
\email{qzfan@shmtu.edu.cn}



\begin{abstract}
Let $0 \xrightarrow{}I \xrightarrow{ } A \xrightarrow{  }A/I
 \xrightarrow{} 0$ be a short exact sequence of C*-algebras.
  In this paper,   we prove that if $\dim({\rm Cu}(I))\leq n$ and $\dim({\rm Cu}(A/I))\leq m$, then $\dim({\rm Cu}(A))\leq n+m+1$,  this result gives an affirmative answer to a problem concerning the  Cu-semigroup of C*-algebras raised by  Thiel and  Vilalta.  We further show that certain comparison and divisibility properties satisfied by  $I$ and $A/I$ are inherited by the  extension algebra $A$. As an application, we provide an alternative proof of equivalence that $A$ is pure if and only if $I$ and $A/I$, a result  originally proven by  Perera,  Thiel, and Vilalta.
\end{abstract}

\maketitle

 \section{Introduction}

Pure ${\rm C^*}$-algebras defined as those satisfying almost divisibility  and strict comparison, were first  introduced by Winter in \cite{W3} to characterize a  class of ${\rm C^*}$-algebras with distinct structural properties and classification theoetic features. In the same work, Winter proved  that pure ${\rm C^*}$-algebras with locally  finite nuclear dimension are $\mathcal{Z}$-stable.  Both almost divisibility and strict comparison are fundamental properties describing the  order structure of Cuntz semigroups, and they play crucial roles in the  classification problems for C*-algebras.

 Pureness is  closely
related to $\mathcal{Z}$-stability:  R{\o}rdam in \cite{RM}  showed that $\mathcal{Z}$-stability implies pureness in full generality. Since pureness lies between strict comparison and $\mathcal{Z}$-stability, the
Toms-Winter conjecture in particular predicts that pureness and $\mathcal{Z}$-stability coincide in this setting.
More broadly, the nonsimple Toms-Winter conjecture (cf. \cite{APTV}) proposes that a separable, nowhere scattered, nuclear ${\rm C^*}$-algebra is $\mathcal{Z}$-stable if and only if it is pure. This has been verified in some cases, for example in \cite{APTV}  and \cite{RT}.

Thiel and  Vilalta introduced the concept of the covering dimension of Cuntz semigroups and established a dimension theory for them (c.f. \cite{HV} and \cite{HV1}). The covering dimension of a Cuntz semigroup of a ${\rm C^*}$-algebra is closely related to its nuclear dimension and topological dimension, bridging the gap between the order theory of Cuntz semigroups and the geometric structure of ${\rm C^*}$-algebras.

The property  of $n$-comparison  was  introduced by Winter in \cite{W3}, while the property  of $m$-almost divisibility  was  introduced by Robert and Tikuisis in \cite{RT}. In \cite{APRTHV3} and \cite{APTV},
 Antoine, Perera,  Robert,  Thiel, and Vilalta established the  $(n,m)$  purity is equivalent to purity in the original sense i.e, almost divisibility  and strict comparison.

The controlled comparison and functionally divisible were introduced by Antoine, Perera,  Thiel and  Vilalta in \cite{APTV}. Functionally divisibility is  a weakly divisible property than $n$-almost divisible.
The same authors also proof that  $m$-comparison satisfies controlled comparison; Controlled comparison together with  functional divisibility implies pureness.

 Fang, Liang, and the author study certain properties of the Cuntz semigroup for quasidiagonal extensions of \(C^*\)-algebras in \cite{LF} and \cite{FFL}.

Let $0 \xrightarrow{}I \xrightarrow{ } A \xrightarrow{  }A/I
 \xrightarrow{} 0$ be a short exact sequence of C*-algebras.
  In this paper,   we prove that if $\dim({\rm Cu}(I))\leq n$ and $\dim({\rm Cu}(A/I))\leq m$, then $\dim({\rm Cu}(A))\leq n+m+1$,  this result gives an affirmative answer to a problem concerning the  Cu-semigroup of C*-algebras raised by  Thiel and  Vilalta in  \cite{HV1}.

In this paper,  we also  prove that if both $I$ and $A/I$ satisfy certain comparison and dibidibility  properties, then the same properties are inherited by $A$. Partial related results were previously established by Perera,  Thiel, and Vilalta
in \cite{PTV}, while we provide an alternative proof in this work.

 As an application of our main results and  Theorem 5.6  proposed by  Seth and  Vilalta in \cite{SV3}, we present a new proof for the equivalence that $A$ is pure if and only if $I$ and $A/I$. This equivalent was  originally proven by  Perera,  Thiel, and Vilalta in \cite{PTV}.

Let $
  0 \xrightarrow{}I \xrightarrow{
 } A \xrightarrow{  }A/I
 \xrightarrow{} 0$ be a short exact extension. Our main results are summarized as follows:

  $(1)$ If $\dim({\rm Cu}(I))\leq n$ and $\dim({\rm Cu}(A/I))\leq m$, then $\dim({\rm Cu}(A))\leq n+m+1$.

 $(1)$ If both  $I$  and $A/I$ possess the  $m$-almost divisible property,  then $A$ has
  weakly ($2$-$m$)-almost divisible.

  $(2)$ If $I$ has the $m$-comparison   and $A/I$ has the $n$-comparison, then $A$ admits the $m+n+1$-comparison.

 $(3)$ If $I$ has weakly  $m_1$-divisible property  and $A/I$ has  weakly  $m_2$-divisible property, then $A$ has
  weakly $m_1+m_2$-divisible.

 $(4)$ If $I$ has weakly  $m_1$-comparison   and $A/I$ has  weakly  $m_2$-comparison property, then $A$ has the
  weakly $m_1+m_2$-comparison.

\section{Preliminaries and definitions}

 Let $A$ be a ${\rm C^*}$-algebra. Given $a, b\in A_+,$
one says  that $a$ is Cuntz subequivalent to $b$ (written $a\precsim b$) if there is a sequence $(v_n)_{n=1}^\infty$
of elements of $A$ such that $$\lim_{n\to \infty}\|v_nbv_n^*-a\|=0.$$
One says that $a$ and $b$ are Cuntz equivalent (written $a\sim b$) if $a\precsim b$ and $b\precsim a$. (See \cite{CJ}.). We  shall write $[a]$ for the equivalence class of $a$.

Let $A$ be a ${\rm C^*}$-algebra.
Suppose $a, b\in {\rm M}_{\infty}(A)_+$. Then $a\in {\rm M}_n(A)_+$ and $b\in {\rm M}_m(A)_+$ for some $n, m\in \mathbb{N}$.
we say that $a$ is  Cuntz subequivalent to $b$  and write  $a\precsim b$ if $a\oplus 0_{\max\{n-m,0\}}\precsim b\oplus 0_{\max\{m-n,0\}}$ as elements in
 ${\rm M}_{\max\{n,m\}}(A)_+$ where $0_n$ is the zero element of ${\rm M}_n(A)_+$.

We define ${\rm W}(A):=M_{\infty}(A)_+/\sim$.

The object ${\rm Cu}(A):={(A\otimes {{\mathcal{K}}}})_+/\sim$
 will be called the Cuntz semigroup of $A$. (See \cite{CEI}.)   ${\rm Cu}(A)$  and ${\rm W}(A)$  become  ordered  semigroups   when equipped with the addition operation
$$[a]+[ b]=[ a \oplus b],$$
  and the order relation
$$[ a]\leq [ b]\Leftrightarrow a\precsim b.$$

More properties of the Cuntz semigroup of C*-algebras  can be found in \cite{APT}, \cite{APRT}, and \cite{GP}.

Given $a$ in ${\rm M}_{\infty}(A)_+$ and $\varepsilon>0,$ we denote by $(a-\varepsilon)_+$ the element of ${\rm C^*}(a)$ corresponding (via the functional calculus) to the function $f(t)={\max (0, t-\varepsilon)},~~ t\in \sigma(a)$. By the functional calculus, it follows in a straightforward manner that $((a-\varepsilon_1)_+-\varepsilon_2)_+=(a-(\varepsilon_1+\varepsilon_2))_+.$

Given any two elements $a, b\in A\otimes \mathcal{K} $, we write $[a]\ll[b]$ if  there exists $\varepsilon>0$ such that $[a]\leq[(b-\varepsilon)_+]$.

Thiel and  Vilalta introduced the concept of the covering dimension of Cu-semigroups and established a dimension theory for them (c.f. \cite{HV1} and \cite{HV2}).

\begin{definition}{\rm ({\rm\cite{HV1}}.)}\label{def:2.1}
Let $S$ be a Cu-semigroup. Given $n\in \mathbb{N}$, we write  $\dim(S)\leq n$,
if whenever $x'\ll x\ll y_1+...+y_r$ in $S$, then there exist $z_{j,k}\in S$ for $j=1, ..., r$ and $k=0, ..., n$ such that

$({\rm i})$ $z_{j,k}\ll y_j$ for each $j$ and $k$;

$({\rm ii})$ $ y\ll \sum_{j,k} z_{j,k}$;

$({\rm iii})$ $\sum_{j=1}^r z_{j,k}\ll x$ for each $k=0,...,n$.

We set $\dim(S)=\infty$ if there exists no $n\in\mathbb{N}$ with $\dim(S)\leq n$. Otherwise, we let $\dim(S)$ be the smallest $n\in\mathbb{N}$ such that $\dim(S)\leq n$. We call $\dim(S)$ the covering dimension of $S$.
\end{definition}

\begin{lemma}\label{lem:2.2}
Let $A$ be a $\mathrm{C}^*$-algebra. Then $\dim(\mathrm{Cu}(A))\leq n$ if and only if for all $[y], [x], [y_1], \dots, [y_r]\in\mathrm{Cu}(A)$ satisfying $[y]\ll[x]\ll[y_1]+\cdots+[y_r]$, for any $\delta>0$ such that $[y]\leq [(x-\delta)_+]$,  there exist $[z_{j,k}]\in\mathrm{Cu}(A)$ for $j=1,\dots,r$ and $k=0,\dots,n$ satisfying the following conditions:
\begin{enumerate}
    \item $[z_{j,k}]\ll [y_j]$ for all $j,k$;
    \item $[(y-\delta)_+]\ll \sum_{j,k}[z_{j,k}]$;
    \item $\sum_{j=1}^r [z_{j,k}]\ll [x]$ for all $k=0,\dots,n$.
\end{enumerate}
\end{lemma}

\begin{proof}$\Longrightarrow$ trivial.
Conversely, assume that for any $[y], [x], [y_1], \dots, [y_r]\in\mathrm{Cu}(A)$ with $[y]\ll[x]\ll[y_1]+\cdots+[y_r]$, we need to show there exist $[z_{j,k}]\in\mathrm{Cu}(A)$ ($j=1,\dots,r$, $k=0,\dots,n$) satisfying conditions \begin{enumerate}
    \item $[z_{j,k}]\ll [y_j]$ for all $j,k$;
    \item $[y]\ll \sum_{j,k}[z_{j,k}]$;
    \item $\sum_{j=1}^r [z_{j,k}]\ll [x]$ for all $k=0,\dots,n$.
\end{enumerate}

Since  $[y]\ll [x]$ in $\mathrm{Cu}(A)$, choose $\delta>0$ such that $[y]\leq [(x-\delta)_+]$. Note that $[(x-\delta/2)_+]\ll [x]\ll [y_1]+\cdots+[y_r]$ and $(x-\delta/2)_+ \ll(x-\delta/4)_+$. By assumption, there exist $[z_{j,k}]\in\mathrm{Cu}(A)$ for $j=1,\dots,r$ and $k=0,\dots,n$ such that
\begin{enumerate}
    \item $[z_{j,k}]\ll [y_j]$ for all $j,k$;
    \item $[y]\leq [(x-\delta)_+]\ll[(x-3\delta/4)_+]\ll \sum_{j,k}[z_{j,k}]$;
    \item $\sum_{j=1}^r [z_{j,k}]\ll [x]$ for all $k=0,\dots,n$.
\end{enumerate}
This completes the argument.
\end{proof}

\begin{lemma}\label{lem:2.3}
Let $A$ be a $\mathrm{C}^*$-algebra. Then $\dim(\mathrm{Cu}(A))\leq n$ if and only if    whenever $[y]\ll [x]\ll [y_1]+...+[y_r]$, where $y, x, y_1, \cdots, y_r$ are all in ${\rm M}_\infty(A)_+$,  then there exist $z_{j,k}\in{\rm M}_\infty(A)_+$ for $j=1, ..., r$ and $k=0, ..., n$ such that

$(1)$ $[z_{j,k}]\ll[y_j]$ for each $j$ and $k$;

$(2)$  $[y]\ll \sum_{j,k} [z_{j,k}]$;

$(3)$ $\sum_{j=1}^r [z_{j,k}]\ll [x]$ for each $k=0,...,n$.
\end{lemma}

\begin{proof}$\Longleftarrow$  For any $[y], [x], [y_1], \dots, [y_r]\in\mathrm{Cu}(A)$ with $[y]\ll[x]\ll[y_1]+\cdots+[y_r]$,
by Lemma \ref{lem:2.2}, for any $\varepsilon >0$, with $[y]\leq [(x-\varepsilon)_+]$, we need to show there exist $[z_{j,k}]\in\mathrm{Cu}(A)$ ($j=1,\dots,r$, $k=0,\dots,n$) satisfying conditions \begin{enumerate}
    \item $[z_{j,k}]\ll [y_j]$ for all $j,k$;
    \item $[(y-\varepsilon)_+]\ll \sum_{j,k}[z_{j,k}]$;
    \item $\sum_{j=1}^r [z_{j,k}]\ll [x]$ for all $k=0,\dots,n$.
\end{enumerate}

Choose sufficiently small $\delta$,  with $\delta<\varepsilon$, such that
$$[(y-\delta/2)_+]\ll [(x-\delta)_+]\ll [(y_1-\delta)_+]+\cdots+[(y_r-\delta)_+].$$
{ Then by Lemma 1.9 in \cite{P3}, } there exists $x', y', y_1', \cdots, y_r'\in {\rm M}_\infty(A)_+$ such that $$[(y-\delta/2)_+]=[y'],~~ [(x-\delta)_+]=[x'],~~ [(y_1-\delta)_+]=[y_1'],~~\cdots [(y_r-\delta)_+]=[y_r'].$$
 Then one has $$[y']\ll [x']\ll [y_1']+\cdots+[y_r'],$$ then  there exist
 exist $z_{j,k}'\in {\rm M}_\infty(A)_+$ ($j=1, \dots, r$, $k=0, \dots, n$)  such that

  $(1')$ $[z_{j,k}']\ll [y_j']$ for all $j,k$;

   $(2')$ $[y']\ll \sum_{j,k}[z_{j,k}']$;

  $(3')$  $\sum_{j=1}^r [z_{j,k}']\ll [x']$ for all $k=0,\dots,n$.

 We take $z_{j,k}=z_{j,k}'$,
therefore, we have

  $(1)$  $[z_{j,k}]\ll [y_j']\ll[y_j]$ for all $j, k$;

  $(2)$  $[(y-\varepsilon)_+]\ll [y']\ll\sum_{j,k}[z_{j,k}]$;

  $(3)$ $\sum_{j=1}^r [z_{j,k}]\ll [x']\ll[x]$ for all $k=0,\dots,n$.

$\Longrightarrow$
  For any $x, y,  y_1', \cdots, y_r'\in {\rm M}_\infty(A)_+$, with $[y]\ll[x]\ll[y_1]+\cdots+[y_r]$.
   Since $\dim(\mathrm{Cu}(A))\leq n$
  there exist
$[z_{j,k}']\in (A\otimes \mathcal{K})_+$ ($j=1,\dots,r$, $k=0,\dots,n$) satisfying conditions

$(1')$ $[z_{j,k}']\ll [y_j]$ for all $j,k$;

$(2')$ $[y]\ll \sum_{j,k}[z_{j,k}']$;

$(3')$ $\sum_{j=1}^r [z_{j,k}']\ll [x]$ for all $k=0,\dots,n$.

 Since $[y]\ll \sum_{j,k}[z_{j,k}']$, there exists $\delta$ such that
  $$[y]\ll \sum_{j,k}[(z_{j,k}'-\delta)_+].$$

 {By Lemma 1.9 in \cite{P3}, } there exist $z_{j,k}\in {\rm M}_\infty(A)_+$ such that $[(z_{j,k}'-\delta)_+]=[z_{j,k}]$,
 therefore one has
 \begin{enumerate}
    \item $[z_{j,k}]\ll[z_{j,k}'] \ll [y_j]$ for all $j,k$;
    \item $[y]\ll \sum_{j,k}[z_{j,k}]$;
    \item $\sum_{j=1}^r [z_{j,k}]\ll \sum_{j=1}^r [z_{j,k}']\ll[x]$ for all $k=0,\dots,n$.
\end{enumerate}

 \end{proof}

By Lemma 3.3  in \cite{HV1} and Lemma  \ref{lem:2.3}, we have the following Theorem.

\begin{lemma}\label{lem:2.4}
Let $A$ be a $\mathrm{C}^*$-algebra. Then $\dim(\mathrm{Cu}(A))\leq n$ if and only if for all $x, y_1, \dots, y_r\in {\rm M}_\infty(A)_+$  satisfying $[x]\ll[y_1]+\cdots+[y_r]$, for any $\delta>0$,  there exist $z_{j,k}\in (A\otimes \mathcal{K})_+$ with $j=1,\dots,r$ and $k=0,\dots,n$ satisfying the following conditions:
\begin{enumerate}
    \item $[z_{j,k}]\leq  [y_j]$ for all $j,k$;
    \item $[(x-\delta)_+]\leq  \sum_{j,k}[z_{j,k}]$;
    \item $\sum_{j=1}^r [z_{j,k}]\leq  [x]$ for all $k=0,\dots,n$.
\end{enumerate}
\end{lemma}
\begin{proof}$\Longleftarrow$ For any $[y], [x], [y_1], \dots, [y_r]\in\mathrm{Cu}(A)$ with $[y]\ll[x]\ll[y_1]+\cdots+[y_r]$, there exists $\delta>0$ such that $$[y]\ll[(x-\delta)_+]\ll[x]\ll[y_1]+\cdots+[y_r].$$  By Lemma \ref{lem:2.3},  there exist $z_{j,k}\in (A\otimes \mathcal{K})_+$ for $j=1,\dots,r$ and $k=0,\dots,n$ satisfying the following conditions:
\begin{enumerate}
    \item $[z_{j,k}]\ll [y_j]$ for all $j,k$;
    \item $[(x-\delta)_+]\ll  \sum_{j,k}[z_{j,k}]$;
    \item $\sum_{j=1}^r [z_{j,k}]\ll [x]$ for all $k=0,\dots,n$.
\end{enumerate} By Lemma 3.3  in \cite{HV1}, one has
\begin{enumerate}
    \item $[z_{j,k}]\leq  [y_j]$ for all $j,k$;
    \item $[y]\leq [(x-\delta)_+]\leq  \sum_{j,k}[z_{j,k}]$;
    \item $\sum_{j=1}^r [z_{j,k}]\leq  [x]$ for all $k=0,\dots,n$.
\end{enumerate}
$\Longrightarrow$ By Lemma 3.3  in \cite{HV1} and Lemma \ref{lem:2.3}.

\end{proof}

\begin{definition}\label{def:2.5}Let $A$ be a  ${\rm C^*}$-algebra.  Let  $a, b\in {\rm {Cu}}(A)$, write $a<_{s}b$ if $(k+1)a\leq kb$ for some nonzero integer $k\in \mathbb{N}$. write $x \leq_N y$ for some $N\in \mathbb{N}$
if $na\leq nb$ for $N\leq n$. \end{definition}

The property  of $n$-comparison  was  introduced by Winter in \cite{W3}.

 \begin{definition}{\rm (\cite{W3}.)}\label{def:2.6}
 Let $A$ be a unital $C^*$-algebra. We say  ${\rm Cu}(A)$ {\rm (${\rm W}(A)$)} has $n$-comparison property,  if where $x,~ y_0,~ y_1,~ y_2,~  \cdots,~ y_n$  are elements
 in ${\rm Cu}(A)$ {\rm (${\rm W}(A)$)} such that $ x<_s  y_j$ for all $j=0,~ 1,~ \cdots, ~ n$, then  $x\leq y_0+y_1+\cdots +y_n$. Here $ x<_s  y$ means that $(k+1)x\leq ky$
 for some nature number $k$. It follows immediately from the definitions that ${\rm Cu}(A)$ {\rm (${\rm W}(A)$)} is almost unperforated if and only if ${\rm Cu}(A)$ {\rm (${\rm W}(A)$)} has 0-comparison.
\end{definition}

The property  of $m$-almost divisible  was  introduced by Robert and Tikuisis in \cite{RT}.

\begin{definition}{\rm (\cite{RT}.)}\label{def:2.7}
	Let $A$ be a unital C*-algebra and let $\operatorname{Cu}(A)$ be its Cuntz semigroup. We say that $\operatorname{Cu}(A)$ is \emph{$m$-almost divisible} for some $m \in \mathbb{N}$ if, for every $a, a' \in \operatorname{Cu}(A)$ with $a'\ll a$  and every $k \in \mathbb{N}$, there exists $b \in \operatorname{Cu}(A)$ such that $k b \leq a$ and $a' \leq (k+1)(m+1)b.$  We say that $A$ is \emph{$m$-almost divisible} if ${\rm Cu}(A)$  is  \emph{$m$-almost divisible}. We say
that $A$ is almost divisible if $\operatorname{Cu}(A)$ is \emph{$0$-almost divisible}.
\end{definition}

The following definition was introduced by Elliott, Fan and Fang in \cite{EFF1}.

\begin{definition}{\rm (\cite{EFF1}.)} \label{def:2.8}  Let $A$ be a unital C*-algebra and let $m\in \mathbb{N}$. We say that ${\rm Cu}(A)$ is weakly ($2$-$m$)-almost divisible if for each $a'\ll a,$  any $k\in \mathbb{N}$  there exists $b\in {\rm Cu}(A)$ such that $kb\leq 2a$ and $a'\leq (k+1)(m+1)b$.  We say that $A$ is weakly ($2$-$m$)-almost divisible if ${\rm Cu}(A)$  is  weakly ($2$-$m$)-almost divisible.
\end{definition}

The following two theorems were obtained by Archey,  Buck,  Mohammadkarimi,  Phillips, and  Seth  in \cite{AJJN}.

\begin{theorem}{\rm (\cite{AJJN}.)}\label{thm:2.9}
Let $A$ be a
${\rm C^*}$-algebra and let $n\in \mathbb{Z}_{>0}$. Then ${\rm Cu}(A)$ has $n$-comparison if and only if ${\rm W}(A)$ has   $n$-comparison.
\end{theorem}

\begin{theorem}{\rm (\cite{AJJN}.)}\label{thm:2.10}
Let $A$ be a
${\rm C^*}$-algebra and let $n\in \mathbb{Z}_{>0}$. Then ${\rm Cu}(A)$ is $m$-almost divisible if and only if  for every $a\in {\rm M}_{\infty}(A)_+$ , every $n\in \mathbb{Z}_{>0}$ and every $\varepsilon>0$, there exist $b\in  {\rm M}_{\infty}(A)_+$ such that in ${\rm W}(A)$ we have $k[b]\leq [a]$ and $[(a-\varepsilon)_+]\leq (k+1)(m+1)[b]$.
 \end{theorem}
The proof of the following theorem is nearly identical to that of Theorem \ref{thm:2.10}, and is therefore omitted.
\begin{theorem}\label{thm:2.11}
Let $A$ be a
${\rm C^*}$-algebra and let $m\in \mathbb{Z}_{>0}$. Then ${\rm Cu}(A)$ is weakly  ($2$-$m$)-almost divisible  if and only if  for every $a\in {\rm M}_{\infty}(A)_+$ , every $k\in \mathbb{Z}_{>0}$ and every $\varepsilon>0$, there exist $b\in  {\rm M}_{\infty}(A)_+$ such that in ${\rm W}(A)$ we have $k[b]\leq 2[a]$ and $[(a-\varepsilon)_+]\leq (k+1)(m+1)[b]$.

 \end{theorem}

 Robert and R{\o}rdam introduced the $(k,w)$-divisible property in \cite{RR}.

\begin{definition}{\rm (\cite{RR}.)} \label{def:2.12} Let $A$ be  a ${\rm C^*}$-algebra  and fix $u\in {\rm Cu}(A)$. Let $m$ be a positive  integer. $u$ is $(m,\omega)$-divisible if for every $u'\in {\rm Cu}(A)$ with $u'\ll u$ there exists $x\in {\rm Cu}(A)$ such that $mx\leq u$ and
 $u'\leq nx$  for some positive integer $n\in \mathbb{N}$.
We say that $A$ is $(m,\omega)$-divisible if  for every  $u\in {\rm Cu}(A)$  is $(m,\omega)$-divisible.
  \end{definition}

The following two theorems  were obtained by Thiel and Vilalta in \cite{HV}.

\begin{theorem} {\rm (\cite{HV}.)} \label{thm:2.13} Let $A$ be  a ${\rm C^*}$-algebra. Then, the following are equivalent:

$(1)$ $A$ has the Global Glimm Propety;

$(2)$ for each $a\in A_+$ and each $\varepsilon>0$, there exists a $*$-homomorphism
$$\varphi:{\rm M}_2(C_0(0,1])\to \overline{aAa}$$ such that $(a-\varepsilon)_+$ belongs to the ideal generated by the image of $\varphi$;

$(3)$ ${\rm Cu}(A)$ is $(2,\omega)$-divisible;

$(4)$ ${\rm Cu}(A)$ is $(n,\omega)$-divisible for every $n\geq 2$;
\end{theorem}

\begin{theorem} {\rm (\cite{HV}.)} \label{thm:2.14} The Global Glimm property passes to hereditary Sub- ${\rm C^*}$-algebras and quotients.
Further, if $I$ is an ideal in a  ${\rm C^*}$-algebra $A$, then $A$ has the Global Glimm property if and only if $I$ and $A/I$ do.

\end{theorem}

The following theorem was obtained by Seth and  Vilalta in  \cite{SV3}.
\begin{theorem}{\rm (\cite{SV3}.)}\label{thm:2.15}
Let \( A \) be a ${\rm C^*}$-algebra, and let \( m, M \in \mathbb{N} \). Assume that \( A \) has \( m \)-comparison and that, for every $N\ge 1$ and for every pair of elements \( x', x \in {\rm{Cu}}(A) \) such that \( x' \ll x \), there exists \( y \in {\rm{Cu}}(A) \) satisfying
$x' \ll N y \ll M x$.

Then, the following are equivalent:

    $(1)$ \( A \) has the Global Glimm Property;

   $(2)$ there exists \( L \in \mathbb{N} \) such that, for every pair \( x', x \in {\rm {Cu}}(A) \) with \( x' \ll x \), there exist \( y_0, y_1 \in {\rm{Cu}}(A) \) satisfying
    $y_0 + y_1 \leq x,  \quad x' \ll Ly_0, Ly_1$;

    $(3)$  \( A \) is pure.

\end{theorem}

Building on Theorem \ref{thm:2.15}, we introduce the following definition.

\begin{definition}\label{def:2.16}
	Let $A$ be a unital ${\rm C^*}$-algebra and let $\operatorname{Cu}(A)$ be its Cuntz semigroup. We say that $\operatorname{Cu}(A)$ is weak {$M$-divisible}  if  for every $N\ge 1$ and for every pair of elements \( x', x \in {\rm {Cu}}(A) \) such that \( x' \ll x \), there exists \( y \in {\rm Cu}(A) \) satisfying
$x' \ll N y \ll M x$.
	\end{definition}

The proof of the following theorem is straightforward and thus omitted.

\begin{theorem}\label{thm:2.17}
Let $A$ be a unital C*-algebra. Then  $\operatorname{Cu}(A)$ is weak {$M$- divisible} if and only if  for every $N\ge 1$ and for every pair of elements \( x', x \in {\rm Cu}(A) \) such that \( x'\ll  x \), there exists \( y \in {\rm Cu}(A) \) satisfying
$x' \leq N y \leq  M x$.
\end{theorem}

The proofs of the following theorems, analogous to that of Theorem \ref{thm:2.10}, were established by Archey, Buck, Mohammadkarimi, Phillips, and Seth in \cite{AJJN}, and thus are omitted here.

\begin{theorem}\label{thm:2.18}
Let $A$ be a
${\rm C^*}$-algebra and let $M\in \mathbb{Z}_{>0}$. Then ${\rm Cu}(A)$ is weak {$M$- divisible} if and only if  for every $a\in {\rm M}_{\infty}(A)_+$ , every $N\in \mathbb{Z}_{>0}$ and every $\varepsilon>0$, there exist $b\in  {\rm M}_{\infty}(A)_+$ such that in ${\rm W}(A)$ we have satisfying
$[(a-\varepsilon)_+]\leq N y \leq  M x$.
 \end{theorem}

Building on Proposition 4.2 in \cite{PTV1}, we present the following definition.

\begin{definition}\label{def:2.19} Let $A$ be a unital C*-algebra and let $m\in \mathbb{N}$.
We say that weakly $M$-comparison if for any $x,y\in {\rm Cu}(A)$ such that $x<_{s}y$ implies $x\leq My$.
\end{definition}

The proofs of the following theorem, analogous to that of Theorem \ref{thm:2.9}, were established by Archey, Buck, Mohammadkarimi, Phillips, and Seth in \cite{AJJN}, and thus are omitted here.

\begin{theorem}\label{thm:2.20}
Let $A$ be a
${\rm C^*}$-algebra and let $n\in \mathbb{Z}_{>0}$.  Then ${\rm Cu}(A)$ has  weakly $M$-comparison if and only if ${\rm W}(A)$ has   weakly $M$-comparison.
\end{theorem}

\section{The main results}

\begin{lemma}	\label{lem:3.1}
Let $0\to I\xrightarrow{i}A\xrightarrow{\pi}A/I\to0$ be a short exact sequence of  ${\rm C^*}$-algebras, and let $\{e_k\}_{k\in \mathbb{N}}$ be a quasi-central approximate identity  of $A$, i.e., $\{e_k\}$ is an approximate identity of $I$ and such that $\|e_kx-xe_k\|\to 0$ for all $a\in A$. For any $a, b\in A_+$ and $m, n\in \mathbb{N}$ satisfying $m[\pi(a)]\leq n[\pi(b)]$ in $A/I$, then for any $\varepsilon>0$, there exists a sufficiently large integer $k$ such that
\[
((1-e_k)^{\frac12}a(1-e_k)^{\frac12}-\varepsilon)_+\otimes I_m
\precsim ((1-e_k)^{\frac12}b(1-e_k)^{\frac12}-\tfrac\varepsilon4)_+\otimes I_n.
\]
Equivalently,
\[
m[((1-e_k)^{\tfrac12}a(1-e_k)^{\tfrac12}-\varepsilon)_+]
\leq n[((1-e_k)^{\tfrac12}b(1-e_k)^{\tfrac12}-\tfrac\varepsilon4\big)_+].
\]

\end{lemma}

\begin{proof} For any $\varepsilon>0$, and  $m[\pi(a)]\leq n[\pi(b)]$.
 We may assume that $m\geq n$,  there exists $v=(v_{ij})\in {\rm M}_{m}(A)$ and a sufficiently small $\delta>0$ with $4\delta<\varepsilon$, satisfying
\[
\|\pi(v)(a\otimes I_m)\pi(v^*)-\pi(b\otimes I_n)\|<\delta,
\]
which is equivalent to
\[
\|\pi(v(a\otimes I_m)v^*-b\otimes I_n)\|<\delta.
\]
Since $\{e_k\}_{k\in \mathbb{N}}$ be a quasi-central   approximate identity of $A$.
We have  $$\lim\limits_{k\to\infty}\|e_kx-xe_k\|=0$$ for all $x\in A$.

 We can choose  sufficiently large $k$ and sufficiently small $\delta'>0$ with $4\sum_{i,j}\|v_{i,j}\|\delta'$ $<\delta$,
 and $4\delta'\leq \delta$. Moreover, choose another small constant $\delta''>0$ such that  $\|e_kx-xe_k\|<\delta''$ implies
$$\|e_k^{\tfrac12}x-xe_k^{\tfrac12}\|<\delta'.$$

Recall the quotient norm identity $$\|\pi(a)\|=\lim\limits_{k\to\infty}\|(1-e_k)a\|,$$ hence
\[
\big\|\pi\big(v(b\otimes I_m)v^*-a\otimes I_n\big)\big\|
=\lim_{k\to\infty}\big\|(1-e_k)\otimes I_m\big(v(b\otimes I_m)v^*-a\otimes I_n\big)\big\|.
\]
Therefore, we may take large enough $k$ such that
\[
\big\|(1-e_k)\otimes I_m\big(v(b\otimes I_m)v^*-a\otimes I_n\big)\big\|<\delta.
\]

Combining the estimate $\big\|(1-e_k)^{\tfrac12}x-x(1-e_k)^{\tfrac12}x\big\|<\delta'$, we compute the norm bound:
\[
\begin{aligned}
&\|v[(1-e_k)^{\tfrac12}b(1-e_k)^{\tfrac12}\otimes I_m]v^*
-(1-e_k)^{\tfrac12}a(1-e_k)^{\tfrac12}\otimes I_n\| \\
<&\,4\|v\|\delta'+\delta<2\delta.
\end{aligned}
\]
It further yields
$$
\|v[((1-e_k)^{\tfrac12}b(1-e_k)^{\tfrac12}\otimes I_m-\tfrac\varepsilon4)_+]v^*
-(1-e_k)^{\tfrac12}a(1-e_k)^{\tfrac12}\otimes I_n\|$$
$$<2\delta+1/4\varepsilon<1/2\varepsilon.$$

By the standard perturbation result for Cuntz comparison from small norm difference, we obtain
\[
[((1-e_k)^{\tfrac12}a(1-e_k)^{\tfrac12}\otimes I_n-\tfrac\varepsilon2)_{+}]\leq
[((1-e_k)^{\tfrac12}b(1-e_k)^{\tfrac12}\otimes I_m-\tfrac\varepsilon4)_+].
\]
Therefore one has  the desired inequality
\[
[((1-e_k)^{\tfrac12}a(1-e_k)^{\tfrac12}-\varepsilon\big)_+\otimes I_m]\leq
 [((1-e_k)^{\tfrac12}b(1-e_k)^{\tfrac12}-\tfrac\varepsilon4\big)_+\otimes I_n].
\]

\end{proof}

\begin{lemma}\label{lem:3.2}
Let $0\to I\xrightarrow{i}A\xrightarrow{\pi}A/I\to 0$ be a short exact sequence of ${\rm C^*}$-algebras, and let $\{e_k\}_{k\in \mathbb{N}}$ be a quasi-central approximate identity  of $A$. For any $a,b\in A_+$ and $m,n\in\mathbb{N}$ with $m[a]\leq n[b]$, then for arbitrary $\varepsilon>0$, there exists a sufficiently large integer $k$ such that the following inequality holds in $I$:
\[
m \left[(e_k^{\frac12}ae_k^{\frac12}-\varepsilon)_+\right]
\leq  n\left[(e_k^{\frac12}be_k^{\frac12}-\frac{\varepsilon}{4})_+\right].
\]

\end{lemma}

\begin{proof} For any $\varepsilon>0$ and  $m[a]\leq n[b]$,
we may assume that $m\geq n$, there exists  $v=(v_{ij})\in {\rm M}_{m}(A)$ and a sufficiently small $\delta>0$ with $4\delta<\varepsilon$ satisfying
$$\|v(b\otimes I_n)v^*-a\otimes I_m\|<\delta.$$

Since $\{e_k\}_{k\in \mathbb{N}}$ be a quasi-central   approximate identity of $A$.
We have  $$\lim\limits_{k\to\infty}\|e_kx-xe_k\|=0$$ for all $x\in A$.

 We can choose  sufficiently large $k$ and sufficiently small $\delta'>0$ with $4\sum_{i,j}\|v_{i,j}\|\delta'$ $<\delta$, and $2\delta'\leq \delta$.
Moreover, choose another small constant $\delta''>0$ such that  $\|e_kx-xe_k\|<\delta''$ implies
$$\|e_k^{\tfrac12}x-xe_k^{\tfrac12}\|<\delta'.$$

Since $\|v(b\otimes I_n)v^*-a\otimes I_m\|<\delta$, one has
$$\|e_k^{\frac12}\otimes I_m(v(b\otimes I_n)v^*-a\otimes I_m)e_k^{\frac12}\otimes I_m\|<\delta,$$
Since $\|e_k^{\tfrac12}x-xe_k^{\tfrac12}\|<\delta',$
this  yields
$$
\|v\big(e_k^{\frac12}be_k^{\frac12}\otimes I_n)v^* - e_k^{\frac12}ae_k^{\frac12}\otimes I_m\| < \delta+2\delta'<2\delta.
$$

Since $\{e_k\}$ is an approximate unit of $I$, we have
\[
\lim_{k\to\infty}\Big\|(e_k\otimes I_m)\big(e_k^{\frac12}be_k^{\frac12}\otimes I_n\big)-e_k^{\frac12}be_k^{\frac12}\otimes I_n\Big\|=0.
\]
When $k$ is large enough,
\[
\Big\|v(e_k\otimes I_m)\big(e_k^{\frac12}b(e_k^{\frac12}\otimes I_m)\big)(e_k\otimes I_n)v^* - e_k^{\frac12}ae_k^{\frac12}\otimes I_m\Big\|<4\delta.
\]

Thus inside $I$ we obtain
\[
[(e_k^{\frac12}ae_k^{\frac12}\otimes I_m-1/2\varepsilon)_+]\leq[(e_k^{\frac12}be_k^{\frac12}-\frac\varepsilon4)_+\otimes I_n],
\]
Therefore, we have
\[
m[(e_k^{\frac12}ae_k^{\frac12}-\varepsilon)_+]
\leq n[(e_k^{\frac12}be_k^{\frac12}-\frac{\varepsilon}{4})_+].
\]
\end{proof}

\begin{theorem}\label{thm:3.3}
		
	Let $0\to I\xrightarrow{i}A\xrightarrow{\pi}A/I$ be a short exact sequence of ${\rm C^*}$-algebras.
 Suppose that $\dim({\rm {Cu}}(I))\le n$ and $\dim({\rm {Cu}}(A/I))\le m$. Then
	\[
	\dim{\rm({Cu}}(A))\le \dim{\rm({Cu}}(I))+\dim{\rm({Cu}}(A/I))+1\leq m+n+1.
	\]
	
\end{theorem}
	
\begin{proof}By Lemma  \ref{lem:2.4},  we need to show  that for any
 $x, y_1, y_2, \dots, y_r\in M_\infty(A)_+$ satisfying $[x]\ll [y_1]+[y_2]+\dots+[y_r]$ and any $\varepsilon>0$, there exist elements $z_{jl}\in (A\otimes \mathcal{K})_+$ ($j=1,\dots,r,\ l=0,1,\dots,n+m+1$) such that the following  condition hold:
	
	$(1)z_{jl}\precsim y_j,$ for all $j=1,\dots,r$ and $l=0, 1, \dots, n+m+1$,
	
	$(2)[(x-\varepsilon)_+]\leq \sum_{j,k} [z_{jl}],$
	
	$(3)\sum_{j=1}^r [z_{jl}]\leq [x]$ for all $l=0, 1, \dots, n+m+1$.
	
	We may assume that all  $x, y_1, y_2, \dots, y_r $ are all in $A_+$ and  $ y_1, y_2, \dots, y_r$ are all orthogonal.

	Since $x\ll y_1+y_2+\dots+y_r$, we have
	\[
	[\pi(x)]\ll [\pi(y_1)]+[\pi(y_2)]+\dots+[\pi(y_r)].
	\]

	In the quotient ${\rm C^*}$-algebra $A/I$, since  $\mathrm{Cu}(A/I)\leq m$, there exist $ z_{jl}'\in (A/I\otimes \mathcal{K})_+$ satisfying
	
	$(1)' $  $\pi(z_{jl}')\precsim \pi(y_j),$ $j=1,\dots,r$ and $l=0,1,\dots,m$,
	
	$(2)'$  $[\pi((x-\varepsilon)_+)]\leq \sum_{j,l}[\pi(z_{jl}')],$
	
	$(3)'$  $ \sum_{j=1}^r[\pi (z_{jl}')]\leq [\pi(x)]$, $l=0,1,\dots,m$,
	
	Let $\{e_k\}\subset I$ be a quasi-central   approximate identity of $A$, i.e., $$\lim\limits_{k\to\infty}\|e_k a-ae_k\|=0$$ for all $a\in A$.

 By Lemma \ref{lem:3.1}, there exist sufficiently large integer $k$, the following Cuntz subequivalence relations hold:

	$(1'')$ $((1-e_k)^{\frac12}z'_{jl}(1-e_k)^{\frac12}-\varepsilon)_+\precsim
 ((1-e_k)^{\frac12}y_j(1-e_k)^{\frac12}-\tfrac\varepsilon4\big)_+$ for $j=1,2,\cdots, r$ and $l=0,1,\dots,m$,

$(2'')$ $[((1-e_k)^{\frac12}(x-4\varepsilon)_+(1-e_k)^{\frac12}-\varepsilon)_+]\leq  \sum_{j,l} [((1-e_k)^{\frac12}z'_{jl}(1-e_k)^{\frac12}-\varepsilon)_+],$
for $l=0, 1, \dots, m$,

$(3'')$ $\sum_{j=1}^r[((1-e_k)^{\frac12}z'_{jk}(1-e_k)^{\frac12}-\varepsilon)_+]\leq
	[((1-e_k)^{\frac12}x(1-e_k)^{\frac12}-\tfrac\varepsilon4)_+].$

Since $[x]\ll [y_1]+[y_2]+\dots+[y_r]$, there exist sufficiently small $\delta$ such that
$$[x]\leq  [(y_1-\delta)_+]+[(y_2-\delta)_+]+\dots+[(y_r-\delta)_+].$$

By Lemma \ref{lem:3.2},  there exist sufficiently large $k$, such that
	
		\[
	\big(e_k^{\frac12}xe_k^{\frac12}-\varepsilon)_+
	\precsim (e_k^{\frac12}y_1e_k^{\frac12}-\tfrac\varepsilon4)_+\oplus (e_k^{\frac12}y_2e_k^{\frac12}-\tfrac\varepsilon4)_+\oplus\cdots\oplus \big(e_k^{\frac12}y_re_k^{\frac12}-\tfrac\varepsilon4\big)_+.
	\]
	
	Since $\dim\mathrm{Cu}(I)\le n$, there exist  $z_{jl}''\in (I\otimes \mathcal{K})_+$,  $j=1,\dots,r$ and $l=0,1,\dots,n$, such that
	
	$(1''')z_{jl}''\precsim (e_k^{\frac12}y_j e_k^{\frac12}-\frac{1}{4r}\varepsilon)_+,$ for all $j=1,\dots,r$ and $l=0,1,\dots,n$,
	
	$(2''')[(e_k^{\frac12}xe_k^{\frac12}-2\varepsilon)_+]\leq \sum_{j,l}[z_{jk}''],$
	
	$(3''')\sum_{j=1}^k[z_{jl}'']\leq[(e_k^{\frac12}xe_k^{\frac12}-\tfrac\varepsilon4)_+]$ for all
	$l=0,1,\dots,n$.
	
	We take $z_{jl}''$,  for $j=1,\dots,r$ and $l=0,1,\dots,n$ and $((1-e_k)^{\frac12}z'_{jl}(1-e_k)^{\frac12}-4\varepsilon)_+$  for $j=1,\dots,r$ and $l=0,1,\dots,m$.

 We have the following:
	
	$(1)$  $z_{jl}''\precsim (e_k^{\frac12}y_j e_k^{\frac12}-1/4r\varepsilon)_+\precsim y_j$, and $ ((1-e_k)^{\frac12}z'_{jl}(1-e_k)^{\frac12}-4\varepsilon)_+ \precsim ((1-e_k)^{\frac12}y_j(1-e_k)^{\frac12}-\varepsilon)_+\precsim y_j,$
	
$(2)$  $[(x-8\varepsilon)_+]\leq [(e_k^{\frac12}xe_k^{\frac12}-2\varepsilon)_+]
	+[((1-e_k)^{\frac12}x(1-e_k)^{\frac12}-4\varepsilon)_+]\leq \sum_{j,k}[z_{jk}'']+
\sum_{j,k}[((1-e_k)^{\frac12}z'_{jk}(1-e_k)^{\frac12}-\varepsilon)_+]$,

	$(3)$ $\sum_{j=1}^l[z_{jk}'']\leq [(e_k^{\frac12}xe_k^{\frac12}-\varepsilon)_+\leq[x]$, and 	$\sum_{j,k}[((1-e_k)^{\frac12}z'_{jk}(1-e_k)^{\frac12}-4\varepsilon)_+]\leq [x].$

The proof is complete.
	
\end{proof}

The following Theorem was obtained by Perera,  Thiel, and Vilalta  in \cite{PTV}, we give a new proof of  this result.

\begin{theorem}\label{thm:3.4}

Let $0\to I\xrightarrow{i}A\xrightarrow{\pi}A/I\to0$ be a short exact sequence of ${\rm C^*}$-algebras. If the ideal $I$ has the $m$-comparison property and the quotient algebra $A/I$ has  the $n$-comparison property, then $A$ satisfies  the $(m+n+1)$-comparison property.

\end{theorem}

\begin{proof}By  Theorem \ref{thm:2.9},  we only need to show that $W(A)$ has $(m+n+1)$-comparison property.
For any  positive elements $x, y_0, y_1, \dots, y_m$, $z_0$, $z_1, \dots, z_n$ $\in {\rm M}_\infty(A)_+$ there exist $k_i$ and $l_j$ such that $(k_i+1)[x]\leq k_i[y_i]$ and
$(l_j+1)[x]\leq l_i[y_i]$,  for any $\varepsilon>0$ we will show that $[(x-\varepsilon)_+]\leq \sum_{i=0}^n[y_i]+\sum_{j=0}^m[z_j]$.

We may take a sufficiently large integer $l$ such that $$(l+1)[x]\leq l[z_j], ~~~~
(l+1)[x]\leq l[y_i],$$ for $j=0,1,\cdots, n$ an $i=0,1,\cdots, m$.

 Since  $(l+1)[x]\leq  l[{z}_j]$, one has  $$(l+1)[\pi(x)]\leq l[\pi({z}_j)],$$  for all $j=0,1,\cdots, n$.
  Since $A/I$ has $n$-comparison property, we have $$[\pi(x)]\leq \sum_{j=1}^n[\pi(z_j)].$$

   Let $\{e_k\}\subset I$ be a quasi-central  approximate identity of $A$.
By Lemma \ref{lem:3.1},
  there exists a sufficiently large index $k$, such that
$$
[((1-e_k)^{\frac12}x(1-e_k)^{\frac12}-\varepsilon)_+]
\leq \sum_{j=0}^n [((1-e_k)^{\frac12}z_j(1-e_k)^{\frac12}-\tfrac\varepsilon4)_+].
$$

Next, since  $(l+1)[x]\leq l[y_i]$, for all $j=0,1,\cdots, m$, by Lemma \ref{lem:3.2}, there exists sufficiently large $k$ such that
\[
(l+1)[(e_k^{\frac12}xe_k^{\frac12}-\varepsilon)_+]
\leq l[(e_k^{\frac12}y_ie_k^{\frac12}-\tfrac\varepsilon4)_+].
\]
Combined with the $m$-comparison property of the ideal $I$, we obtain
\[
[(e_k^{\frac12}xe_k^{\frac12}-\varepsilon)_+]
\leq \sum_{i=0}^m [(e_k^{\frac12}y_ie_k^{\frac12}-\tfrac\varepsilon4)_+].
\]

We use the standard decomposition inequality for positive truncations:
\[
(x-3\varepsilon)_+
\precsim  (e_k^{\frac12}xe_k^{\frac12}-\varepsilon)_+
+((1-e_k)^{\frac12}x(1-e_k)^{\frac12}-\varepsilon)_+.
\]
Substitute the two comparison inequalities derived above:
$$[(x-3\varepsilon)_+]\leq \sum_{i=0}^m [(e_k^{\frac12}y_ie_k^{\frac12}-\varepsilon)_+]
+\sum_{j=0}^n [((1-e_k)^{\frac12}z_j(1-e_k)^{\frac12}-\varepsilon)_+]$$
$$\leq \sum_{i=0}^m [y_i]+\sum_{j=0}^n [z_j].$$

By arbitrariness of $\varepsilon>0$, we have $[x]\leq \sum_{i=0}^m [y_i]+\sum_{j=0}^n [z_j]$, which means $A$ possesses the $(m+n+1)$-comparison property.

\end{proof}

\begin{theorem}\label{thm:3.5}
	
Let $0\to I\xrightarrow{i}A\xrightarrow{\pi}A/I\to 0$ be a short exact sequence of ${\rm C^*}$-algebras. Suppose the ideal $I$ has weak $m$-comparison property and the quotient algebra $A/I$ has  weak $n$-comparison property. Then $A$ satisfies the weak $(m+n)$-comparison property.

\end{theorem}
	
\begin{proof}
By Theorem \ref{thm:2.20}, for any $\varepsilon>0$ and any integer $l$  and any  $x, y\in {\rm M}_\infty(A)_+$ satisfying $(l+1)[x]\leq l[y]$.
 we only need to show $$[(x-\varepsilon)_+]\leq(m+n)[y].$$

The quotient map $\pi$ preserves Cuntz comparison, which implies
\[
(l+1)[\pi(x)]\leq l[\pi(y)].
\]
By the weak $n$-comparison property of $A/I$, we have $$[\pi(x)]\leq  n[\pi(y)].$$

Let $\{e_k\}\subset I$ be a quasi-central   approximate identity of $A$. Applying Lemma \ref{lem:3.1},  there exists a sufficiently large integer $k$ such that,
\[
[((1-e_k)^{\tfrac12}x(1-e_k)^{\tfrac12}-\varepsilon)_+]\leq  n[((1-e_k)^{\tfrac12}y(1-e_k)^{\tfrac12}-\tfrac\varepsilon4)_+].
\]

Since $(l+1)[x]\leq l[y]$, apply  Lemma \ref{lem:3.2}, there exist  sufficiently large $k$ satisfying
\[
(l+1)[(e_k^{\tfrac12}xe_k^{\tfrac12}-\varepsilon)_+]
\leq  l[(e_k^{\tfrac12}ye_k^{\tfrac12}-\tfrac\varepsilon4)_+].
\]
Since $I$ has the weak $m$-comparison property, we obtain
\[
[(e_k^{\tfrac12}xe_k^{\tfrac12}-\varepsilon)_+]\leq  m[(e_k^{\tfrac12}ye_k^{\tfrac12}-\tfrac\varepsilon4)_+].
\]

Use the truncation decomposition inequality for positive elements:
\[
(x-3\varepsilon)_+
\precsim  (e_k^{\tfrac12}xe_k^{\tfrac12}-\varepsilon)_+
+((1-e_k)^{\tfrac12}x(1-e_k)^{\tfrac12}-\varepsilon)_+.
\]
Substitute the two comparison inequalities:

$$[(x-3\varepsilon)_+]\leq
m[(e_k^{\tfrac12}ye_k^{\tfrac12}-\tfrac\varepsilon4)_+]+n[((1-e_k)^{\tfrac12}y(1-e_k)^{\tfrac12}-\tfrac\varepsilon4)_+]$$
$$\leq m[y]+n[y]=(m+n)[y].$$

From arbitrariness of $\varepsilon>0$, $[x]\leq (m+n)[y]$, hence $A$ has the weak $(m+n)$-comparison property.

\end{proof}

\begin{theorem}\label{thm:3.6}

Let $0\to I\xrightarrow{i}A\xrightarrow{\pi}A/I\to 0$ be a short exact sequence of  ${\rm C^*}$-algebras. If the ideal $I$ and the quotient algebra $A/I$ have the $m$-divisible  property, then $A$ possesses the weak  ($2$-$m$)-divisible  property.

\end{theorem}

\begin{proof} By Theorem \ref{thm:2.11}, for any $\varepsilon>0$ any $a\in {\rm M}_\infty(A)_+$ and any integer  $n>0$, we need to show that there exist  $b\in (A\otimes \mathcal{K})_+$  such that
\[
n[b]\leq 2[a],\quad [(a-\varepsilon)_+]\leq(n+1)(m+1)[b].
\]
Let $\{e_k\}\subset I$ be a quasi-central approximate identity of $A$. Choose sufficiently large $k$ to get the  decomposition:
$$\big(a-8\varepsilon\big)_+
\precsim (e_k^{\tfrac12}ae_k^{\tfrac12}-2\varepsilon)_+
+((1-e_k)^{\tfrac12}a(1-e_k)^{\tfrac12}-2\varepsilon)_+.$$

Notice $(e_k^{\tfrac12}ae_k^{\tfrac12}-\varepsilon)_+\in I$. By the $m$-divisible  property of $I$, there exists $b_1\in (I\otimes \mathcal{K})_+$ such that
$$n[b_1]\leq [(e_k^{\tfrac12}ae_k^{\tfrac12}-\varepsilon)_+],$$  and $$[(e_k^{\tfrac12}ae_k^{\tfrac12}-2\varepsilon)_+]\leq (n+1)(m+1)[b_1].$$

Since $\pi(a)\in(A/I)_+$,  and $A/I$ has $m$-divisible  property, there exists ${b}_2\in (A\otimes \mathcal{K})_+$
such that  $$k[\pi({b}_2)]\leq [\pi(a)],$$ and $$[(\pi({b}_2)-\varepsilon)_+]\leq (m+1)(n+1)[b_2].$$
 Applying  Lemma \ref{lem:3.1},  pick sufficiently large $k$  such that
$$k[((1-e_k)^{\tfrac12}b_2(1-e_k)^{\tfrac12}-\varepsilon)_+]\leq [((1-e_k)^{\tfrac12}a(1-e_k)^{\tfrac12}-\tfrac\varepsilon4)_+],$$ and
$$[((1-e_k)^{\tfrac12}a(1-e_k)^{\tfrac12}-2\varepsilon)_+\leq (m+1)(n+1)[(b_2-1/4\varepsilon)_+]$$

Set $b:=b_1\oplus ((1-e_k)^{\tfrac12}b_2(1-e_k)^{\tfrac12}-\varepsilon)_+$. Then
\[
\begin{aligned}
kb &= k[b_1]+ k[b_2]\\\
&\leq [(e_k^{\tfrac12}ae_k^{\tfrac12}-\varepsilon)_+]
+[((1-e_k)^{\tfrac12}a(1-e_k)^{\tfrac12}-\tfrac\varepsilon4)_+]\\\
&\leq 2[a],\\\\[4pt]
[(a-8\varepsilon\big)_+]
&\leq [(e_k^{\tfrac12}ae_k^{\tfrac12}-2\varepsilon)_+]
+[((1-e_k)^{\tfrac12}a(1-e_k)^{\tfrac12}-2\varepsilon)_+]\\\
&\leq (m+1)k[b_1]+(n+1)k[b_2]\\\
&\leq (m+1)(n+1)([b_1]+[b_2])\\\
&=(m+1)(n+1)[b].
\end{aligned}
\]

Therefore, $A$ has the weak  ($2$-$m$)-divisible  property.

\end{proof}

\begin{theorem}\label{thm:3.7}
	
	Let $0\to I\xrightarrow{i}A\xrightarrow{\pi}A/I\to 0$ be a short exact sequence of ${\rm C^*}$-algebras. If the ideal $I$ has the weak $m_1$-divisible property, and the quotient algebra $A/I$ has the weak $m_2$-divisible  property, then $A$ possesses the weak  $(m_1+m_2)$-divisible property.
	
\end{theorem}

\begin{proof}By Theorem \ref{thm:2.18}, for any $\varepsilon>0$ any $x\in {\rm M}_\infty(A)_+$ and any integer  $n>0$, we need to show that there exist  $y\in (A\otimes \mathcal{K})_+$
$$[(x-\varepsilon)_+]\leq n[y]\leq (m_1+m_2)[x].$$

Since  $\pi:A\to A/I$  is  the quotient map, then $\pi(x)\in(A/I)_+$. Since $A/I$ has  $m_2$-divisible  property, for fixed  $n\in\mathbb{N}$, there exists $y\in M_\infty(A)_+$ satisfying
$$[\pi(x)]\leq  n[\pi (y)]\leq m_2[\pi(x)].
$$

Let $\{e_k\}\subset I$ be a quasi-central approximate identity of $A$. Applying Lemma \ref{lem:3.1},  there exists a sufficiently large integer $k$ such that:

$$[((1-e_k)^{\frac12}x(1-e_k)^{\frac12}x-16\varepsilon)_+]\leq n[((1-e_k)^{\frac12}y(1-e_k)^{\frac12}-4\varepsilon)_+)]$$
$$\leq  m_2[((1-e_k)^{\frac12}x(1-e_k)^{\frac12}-\varepsilon)_+].$$

Note that $(e_k^{\frac12}xe_k^{\frac12}-\varepsilon)_+\in I$. From the $m_1$-divisible property of $I$, there exists $y'\in I$ with
\[
[(e_k^{\frac12}xe_k^{\frac12}-2\varepsilon)_+]\leq n[y']\leq m_1[\big(e_k^{\frac12}xe_k^{\frac12}-\varepsilon)_+].
\]

Define
$$
y:=y'\oplus ((1-e_k)^{\frac12}y(1-e_k)^{\frac12}-4\varepsilon)_+.
$$
We use the standard decomposition inequality for positive elements:
\[
(x-20\varepsilon)_+
\precsim ((1-e_k)^{\frac12}x(1-e_k)^{\frac12}-16\varepsilon)_+
+(e_k^{\frac12}xe_k^{\frac12}-2\varepsilon)_+.
\]

Therefore, one has
$$[(x-20\varepsilon)_+]\leq[(e_k^{\frac12}xe_k^{\frac12}-2\varepsilon)_+]+ [((1-e_k)^{\frac12}x(1-e_k)^{\frac12}x-16\varepsilon)_+]$$
$$\leq  n[y']+n[\big((1-e_k)^{\frac12}y(1-e_k)^{\frac12}-4\varepsilon)_+]$$$$=n[y]
\leq m_1[(e_k^{\frac12}xe_k^{\frac12}-\varepsilon)_+]
+m_2[((1-e_k)^{\frac12}x(1-e_k)^{\frac12}-\varepsilon)_+]$$$$
\leq m_1[x]+m_2[x]=(m_1+m_2)[x].$$

 Therefore, $A$ has the $(m_1+m_2)$-divisible comparison property.

\end{proof}

The following theorem was originally established by Perera,  Thiel, and Vilalta  in \cite{PTV},  using  the preceding  theorems, we provide an alternative  proof of  this result.
	
\begin{theorem}\label{thm:3.8}  Let $
  0 \xrightarrow{}I \xrightarrow{
 } A \xrightarrow{ }A/I
 \xrightarrow{} 0$ be an extension of ${\rm C^*}$-algebras. Then $A$ is pure if and only if $I$ and the quotient $A/I$ are pure.
 \end{theorem}
\begin{proof} By Theorem \ref{thm:2.15},  a C*-algebra is pure provided it satisfies three condition: it has the  $m$-comparison property for any integer $m$, it is  weakly $n$-divisible for every nonzero integer $n$, and it possesses the Global Glimm Property. We verify these properties for $A$ via the purity of $I$ and $A/I$.

First, since $I$ and $A/I$ are pure,  Theorem \ref{thm:3.4} implies that $A$ has $1$-comparison property.

 Second, as pure C*-algebras  $I$ and $A/I$ are both  $(2,\omega)$-divisible. By Theorem \ref{thm:2.13},  $I$ and $A/I$ enjoy the   Global Glimm Property, which further yields that $A$  has the Global Glimm Property by Theorem \ref{thm:2.14}.

   Third, the purity of  $I$ and $A/I$ also guarantees that both $I$ and $A/I$  are  almost divisible, hence  weakly $m$-divisible for any integer $m\geq 2$. Combining this fact with  Theorem \ref{thm:3.7},  $A$ has  weakly $2m$-divisible.
\end{proof}

\end{document}